\documentclass[11pt, a4paper]{article}
\usepackage{amsmath, amssymb, amsthm}
\usepackage{geometry}
\usepackage{hyperref}
\usepackage{xcolor}

\newtheorem{theorem}{Theorem}[section]
\newtheorem{lemma}[theorem]{Lemma}
\newtheorem{proposition}[theorem]{Proposition}

\newtheorem{example}[theorem]{Example}
\newtheorem{definition}[theorem]{Definition}
\newtheorem{remark}[theorem]{Remark}

\def\C{\mathbb{C}}
\def\F{\mathbb{F}}
\def\K{\mathbb{K}}
\def\FR{\mathcal{R}}

\title{Rosenbrock's Theorem characterizes Pr\"{u}fer domains}
\author{Vanni Noferini\thanks{Supported by a Research Council of Finland grant (decision number 370932).
Department of Mathematics and Systems Analysis, Aalto University, P.O. Box 11100, FI-00076, Aalto, Finland. E-mail: {vanni.noferini@aalto.fi}}
}
\date{11 September 2026}

\begin{document}

\maketitle

\begin{abstract}
Under coprimality assumptions on certain submatrices, Rosenbrock's Theorem relates the Smith form of a matrix $P$ over an elementary divisor domain $\FR$ to the Smith-McMillan form of a matrix $G$ over the field of fractions of $\FR$, where $G$ is a Schur complement in $P$. If $\FR$ is not an elementary divisor domain, Rosenbrock's Theorem is not directly applicable in its original form, because not every matrix is unimodularly equivalent to a matrix in Smith form. In this paper, we state an ideal-theoretic version of Rosenbrock's Theorem that is meaningful over any integral domain, and we show that it is equivalent to the classic formulation over an elementary divisor domain. Moreover, we give a characterization of Pr\"{u}fer domains as those integral domains over which the ideal-theoretic version of Rosenbrock's Theorem holds for every matrix satisfying the assumptions. In particular, the theorem does not hold for every
admissible matrix over $\C[x_1,\dots,x_d]$ when $d\geq 2$.
{However, when $d\leq 3$, it holds for square
matrices in a nonempty Zariski-open subset of the
coefficient space, having fixed the block sizes and a total degree
bound.} Finally, we prove that, if $\FR$ is an integral domain such that every right invertible matrix can be completed to a unimodular matrix, then every matrix $P$ that satisfies the assumptions of the ideal-theoretic Rosenbrock's Theorem and realizes the same Schur complement $G$ shares the same ideal-theoretic generalization of the Smith form.
\end{abstract}


\bigskip


\section{Introduction}

In linear systems theory, Rosenbrock's Theorem is a result that relates the invariant factors of a polynomial system matrix $P$ to those of its ``transfer function" (an engineering term that denotes the Schur complement in $P$ of a leading principal submatrix, assumed invertible) $G$. Motivated by control-theoretic applications, Rosenbrock first proved this theorem \cite[Theorem 4.1, Chapter 3]{Rosenbrock}, whose statement assumes coprimality of certain submatrices of $P$, over the univariate polynomial rings $\mathbb{R}[x]$ and $\mathbb{C}[x]$ and using ad hoc arguments. Rosenbrock's Theorem was later proved in much wider generality, first by Coppel \cite{Coppel1974} for matrices over any principal ideal domain and more recently by Dopico, Noferini and Zaballa \cite{DNZ2025} for matrices over any elementary divisor domain. Because one of the equivalent definitions of an elementary divisor  domain \cite{Friedland2015, Kaplansky1949} is precisely an integral domain over which every matrix is unimodularly equivalent to one in Smith form \cite{Coppel1974,DNZ2025,Friedland2015,Kaplansky1949,Rosenbrock} (and hence every matrix over the associated  field of fractions is unimodularly equivalent to one in Smith-McMillan form \cite{Coppel1974,DNZ2025,Rosenbrock}), the results in \cite{DNZ2025} imply that Rosenbrock's Theorem, in its original version, holds over every ring for which its statement makes sense.

Nonetheless, since the invariant factors in a Smith form are related to the ideals generated by the minors of a given size, the statement of Rosenbrock's Theorem can be translated into the language of ideals. This observation potentially reopens the question of its applicability, that was closed by \cite{DNZ2025} for its classic formulation. There are at least two reasons to study this problem. One is purely algebraic: The ideal-theoretic version of Rosenbrock's Theorem that we propose in this article is a well posed mathematical statement over every integral domain, and thus characterizing those integral domains for which the theorem holds is a sensible mathematical question within ring theory. The second comes once again from systems theory: Multidimensional systems \cite{AGP2025} are modelled by multivariate polynomial rings such as $\mathbb{R}[x_1,\dots,x_d]$ and $\mathbb{C}[x_1,\dots,x_d]$. It is widely known that these rings are not elementary divisor domains, and thus the original formulation of Rosenbrock's Theorem is inapplicable; but one still may wonder (and scholars have wondered) if it can be replaced by some other tool with a similar flavour. 

The main goal of the present manuscript is to construct an ideal-theoretic generalization of Rosenbrock's Theorem, that is equivalent to the original formulation for an elementary divisor domain. We replace the scalar invariant factors by certain ideals, that we call \emph{Smith ideals}, and are defined as ideal quotients of two consecutive determinantal ideals. In turn, Smith ideals can be used to define \emph{numerator and denominator ideals} associated to any matrix over the field of fractions of the base ring.  The main results of our paper include:
\begin{itemize}
\item Theorem \ref{thm:rosenbrock}, which answers the question posed above. Namely, Theorem \ref{thm:rosenbrock} establishes that, given an integral domain $\FR$, the ideal-theoretic Rosenbrock's Theorem holds for every matrix  over $\FR$ satisfying the assumptions if and only if $\FR$ is a Pr\"{u}fer domain \cite{FHP1997,gilmer1968}. This adds a new characterization of Pr\"{u}fer domains to the many already available. From the viewpoint of systems theory, the result is in some sense negative because polynomial rings with coefficients in a field are not Pr\"{u}fer unless $d=1$. Nonetheless, we hope that the theoretical constructions leading to the statement of Rosenbrock's Theorem may still be useful for systems theorists as well, because what Theorem \ref{thm:rosenbrock} states is that, for non-Pr\"{u}fer domains, Rosenbrock's Theorem does not hold for \emph{every} minimal matrix pair $(P,G)$, but it does not a priori forbid zero/pole recovery (in the ideal-theoretic sense described in this paper) for \emph{specific} subclasses of matrix pairs. Potentially, this opens the line of research of determining for which multidimensional system matrices the ideal-theoretic version of Rosenbrock's Theorem holds, or of investigating whether some weaker variant may still apply more generally.
\item Theorem \ref{prop:vectorcase}, which identifies one subclass of matrix pairs for which the ideal-theoretic Rosenbrock's Theorem holds over \emph{every} integral domain: Those such that $G$ is either a row vector or a column vector (or a scalar). In other words, for the theorem to fail, $G$ must have at least two rows and at least two columns.
\item Theorem \ref{prop:easy} provides, over integrally closed domains, a result dual to Theorem \ref{prop:vectorcase}, assuming that $G$ is precisely one row and one column smaller than $P$. Potentially, both Theorems \ref{prop:vectorcase} and \ref{prop:easy} may be applicable to system theory.
\item Proposition \ref{rem:generic}, which states that, if $\FR=\mathbb{F}[x_1,\dots,x_d]$ for an algebraically closed field $\mathbb{F}$ and $d \leq 3$, then Rosenbrock's Theorem holds for almost every square matrix in the Zariski topology {(having fixed the block sizes and a total degree bound)}.
\item Theorem \ref{thm:allParethesame}, a result that may have useful implications for system theory:  Over any integral domain where every right invertible matrix can be completed to a unimodular matrix (including polynomial rings, by the Quillen-Suslin Theorem \cite{Quillen1976}) all minimal realizations (that is, all matrices $P$ satisfying the assumptions of the ideal-theoretic Rosenbrock's Theorem \ref{thm:rosenbrock}) of the same Schur complement $G$ share identical {(up to appending if needed an initial trivial subsequence to the shorter lists)} Smith ideals, even when these do not allow exact recovery of the numerator and denominator ideals of the Schur complement.
\end{itemize}
The article is organized as follows: Section \ref{sec:prel} briefly recalls basic concepts about fractional and inverse ideals, for the benefit of those readers who may not be experts in ring theory. In Section \ref{sec:smith} we introduce and develop essential properties of the Smith ideals for matrices over an integral domain and their extension to matrices in the associated field of fractions, as well as the numerator ideals and the denominator ideals. Finally, Section \ref{sec:rosenbrock} contains the statement and proof of the main results, as well as some examples and remarks.

\section{Preliminaries}\label{sec:prel}

Let us first recall some basic definitions from ring theory \cite{Atiyah1969}.

\begin{definition}
Let $\mathcal{R}$ be a commutative ring. A subset $I \subseteq \mathcal{R}$ is an ideal if it is an additive subgroup of $\FR$ and $r x \in I$ for all $r \in \mathcal{R}$ and $x \in I$.
\end{definition}

An ideal $I$ is called finitely generated if there exist $g_1,\dots,g_\ell \in \mathcal{R}$ such that every $x \in I$ is an $\mathcal{R}$-linear combination of $g_1,\dots,g_\ell$; in this case we write $I=\langle g_1,\dots,g_\ell \rangle$. Observe in particular that $\FR=\langle 1 \rangle$ and that $1 \in I$ if and only if $I=\FR$. If $\FR$ is an integral domain with field of fractions $\K$ and $x \in \K$ and $S \subseteq \FR$ is a subset, we denote $x S := \{ x a \mid a \in S \}$ and $\langle x \rangle := x \FR$.

\begin{definition}
Let $\mathcal{R}$ be an integral domain and $\mathbb{K}$ its field of fractions. An $\mathcal{R}$-module $I\subseteq \mathbb{K}$ is a fractional ideal of $\mathcal{R}$ if there exists a non-zero element $d \in \mathcal{R}$ such that $d I \subseteq \mathcal{R}$.
\end{definition}

Note that if $I$ is a fractional ideal and $d I \subseteq \FR$, then $dI$ is an (integral) ideal of $\FR$.

\begin{definition}
If $I$ and $J$ are two fractional ideals of $\mathcal{R}$, then:
\begin{enumerate}
\item The ideal sum of $I$ and $J$ is $I+J=\{ a + b \mid a \in I, b \in J  \}$;
\item The ideal product of $I$ and $J$ is $I \cdot J := \left\{ \sum_{i=1}^\ell a_i b_i \mid a_i \in I, b_i \in J, 1 \leq \ell \in \mathbb{N} \right\}$;
\item  If $J \neq \langle 0 \rangle$, the ideal quotient of $I$ by $J$, denoted $I:J$, is  $\{ x \in \K \mid xJ \subseteq I \}$;
\item If $I \neq \langle 0 \rangle$, the inverse ideal of $I$ is the ideal quotient $I^{-1}:=\FR : I$. Generally, it is clear by definition that $I \cdot I^{-1} \subseteq \FR$, and the ideal $I$ is called \emph{invertible} if $I \cdot I^{-1} = \FR$.
\end{enumerate}

\end{definition}

Clearly, an ideal quotient is itself a fractional ideal. Indeed, suppose that $x,y \in I:J$ and $r \in \FR$. Then, since $I$ is itself a fractional ideal, it holds $(x+y) J \subseteq I$ and $rx J \subseteq I$. In addition, if $0 \neq d\in \FR$ is such that $d I \subseteq \FR$ and $0 \neq y \in J \cap \FR$ (which always exists when $J \neq \langle 0 \rangle$), it is clear that $(dy)(I:J) \subseteq \FR$. Similarly one can show that $I+J$, $I \cdot J$ and $I^{-1}$ are fractional ideals.

We recall below some useful properties related to inverse ideals \cite{gilmer1968} that will be useful later.

\begin{proposition}\label{prop:invertible}
Let $I,J\neq \langle 0 \rangle$ be fractional ideals of the integral domain $\FR$. It always holds $I \cdot J^{-1} \subseteq I : J$. Moreover, if $J$ is invertible, then $I \cdot J^{-1} = I : J$.
\end{proposition}
\begin{proof}
Let $x \in I \cdot J^{-1}$, then $x=\sum_i a_i b_i$, where $a_i \in I$ and $b_i \in \K$ satisfy $b_i c \in \FR$ for all $c \in J$. Then, $x c = \sum_i a_i (b_i c) \in I$ because it is an $\FR$-linear combination of elements of $I$, and hence $x \in I:J$.

Let now $x \in I :J$ and suppose $J$ is invertible. Then $x J \subseteq I$ and hence $x \FR = x J \cdot J^{-1} \subseteq I \cdot J^{-1}$. It follows that $x \in x \FR \subseteq I \cdot J^{-1}$.
\end{proof}

\begin{lemma}\label{lem:inverse2gen}
Let $\FR$ be an integral domain and let $r_1,r_2 \in \FR$ be such that $r_1 \neq 0 \neq r_2$. Then
\[   \langle r_1, r_2 \rangle^{-1} = \frac{1}{r_1 r_2} ( \langle r_1 \rangle \cap \langle r_2 \rangle).\]
\end{lemma}
\begin{proof}
Set $I:=\frac{1}{r_1 r_2} ( \langle r_1 \rangle \cap \langle r_2 \rangle)$. Suppose $x = \frac{a}{r_2} = \frac{b}{r_1} \in I$ and $y=c r_1 + d r_2 \in \langle r_1,r_2 \rangle$, for some $a,b,c,d \in \FR$. Then $x y = b c + ad \in \FR$, and hence $x \in \langle r_1, r_2 \rangle^{-1}$.

Conversely, suppose that $x \in \langle r_1, r_2 \rangle^{-1}$. Then, $ x r_1 \in \FR$ and hence $r_1 r_2 x = r_2 (x r_1) \in \langle r_2 \rangle$; it can be shown analogously that $r_1$ also divides $r_1 r_2 x$. Hence, $r_1 r_2 x \in \langle r_1 \rangle \cap \langle r_2 \rangle$.
\end{proof}

\section{Smith ideals, denominator ideals, and numerator ideals}\label{sec:smith}
Although the theory that we develop in this section is, in part, also applicable to more general rings, henceforth $\FR$ will denote an integral domain; when additional assumptions are made on $\FR$, they will be stated explicitly. Given a matrix $M \in \mathcal{R}^{m \times n}$, for $k=1,\dots,\min\{m,n\}$ we denote by $\mathcal{I}_k(M)$ the $k$-th determinantal ideal of $M$, or in other words the ideal of $\mathcal{R}$ generated by the $k \times k$ minors of $M$; we also conventionally define $\mathcal{I}_k(M)=\mathcal{R}$ when $k \leq 0$ and $\mathcal{I}_k(M)=\langle 0 \rangle$ when $k > \min \{m,n\}$. These ideals have sometimes also been called the Fitting ideals \cite[Section 20.2]{Eisenbud1995} of $M$. 
\begin{lemma}\label{lem:generalinclusions}
Let $\FR$ be an integral domain and $M \in \mathcal{R}^{m \times n}$. Then, for all $k=1,\dots,\min\{m,n\}$, it holds $\mathcal{I}_{k}(M) \subseteq \mathcal{I}_{k-1}(M) \cdot \mathcal{I}_1(M) \subseteq \mathcal{I}_{k-1}(M)$.
\end{lemma}
\begin{proof}
It follows by the Laplace expansion of determinants \cite[Section 0.8.9]{HoJo}. 
\end{proof}

\begin{definition}[Smith ideals]
Let $M \in \FR^{m \times n}$ have determinantal ideals $\mathcal{I}_k(M)$. For $k=1,\dots,\min\{m,n\}$, if $\mathcal{I}_{k-1}(M) \neq \langle 0 \rangle$ then the $k$-th \emph{Smith ideal} of $M$ is the ideal quotient
\[\mathcal{S}_k(M) : =\mathcal{I}_{k}(M) : \mathcal{I}_{k-1}(M).\] 
If instead $\mathcal{I}_{k-1}(M)=\langle 0 \rangle$, then $\mathcal{S}_k(M) := \langle 0 \rangle$.
\end{definition}

When $\mathcal{R}$ is an elementary divisor domain \cite{DNZ2025,Friedland2015, Kaplansky1949}, and thus in particular a B\'{e}zout domain, then all its finitely generated ideals are principal. In particular, over an elementary divisor domain we have $\mathcal{I}_k(M) = \langle \gamma_k(M) \rangle$ where $\gamma_k(M)$ is the $k$-th determinantal divisor of $M$ (greatest common divisor of the $k \times k$ minors of $M$); this implies in turn $\mathcal{S}_k(M) = \langle s_k(M) \rangle$, where $s_k(M)$ is the $k$-th invariant factor of $M$ (entry in the $(k,k)$ position in a Smith form of $M$). Therefore, Smith ideals can be seen as a generalization of the invariant factors. Unlike invariant factors, the Smith ideals are always well defined even over integral domains that are not elementary divisor domains. Recall that a square matrix with elements in $\FR$ is called \emph{unimodular} if its determinant is a unit of $\FR$. Theorem \ref{thm:unimodularinvariance} shows instead that, like invariant factors, Smith ideals are invariant by unimodular equivalence.

\begin{theorem}\label{thm:unimodularinvariance}
The determinantal ideals (and hence the Smith ideals) of a matrix $M \in \FR^{m \times n}$ are invariant by unimodular equivalence, i.e., $\mathcal{I}_k(M)=\mathcal{I}_k(UMV)$ for all $k$ and all unimodular $U \in \FR^{m \times m}, V \in \FR^{n \times n}$.
\end{theorem}
\begin{proof}
By the Cauchy-Binet formula \cite[Section 0.8.7]{HoJo}, $\mathcal{I}_k(UMV) \subseteq \mathcal{I}_k(M)$. The inverse inclusion follows analogously by the identity $M=U^{-1}(UMV)V^{-1}$.
\end{proof}

Over an elementary divisor domain, the Smith ideals $\mathcal{S}_k(M)$ are integral ideals, finitely generated, and satisfy $\mathcal{S}_{k+1}(M) \subseteq \mathcal{S}_k(M)$ for all $k=1,\dots,\min\{m,n\}-1$. In fact, since in an elementary divisor domain every finitely generated ideal is principal, this translates to the divisibility chain property of invariant factors, $s_k(M) \mid s_{k+1}(M)$. The first property holds true more generally over integrally closed domains, as we show in Theorem \ref{thm:smithintegral}. The second and third properties hold true over a more restrictive class called Pr\"{u}fer domains; we show this in Theorem \ref{thm:smithchain}, and we illustrate in Example \ref{ex:noSmithdivisibility} that it is not generally true even for unique factorization domains.

\begin{theorem}\label{thm:smithintegral}
Let $\FR$ be an integrally closed domain and $M \in \FR^{m \times n}$. Then, for all $k=1,\dots,\min\{m,n\}$, the Smith ideals $\mathcal{S}_k(M)$ are integral ideals of $\FR$.
\end{theorem}
\begin{proof}
{If $\mathcal{I}_{k-1}(M)= \langle0\rangle$, the statement is obvious. Otherwise,} take an arbitrary $x \in \mathcal{S}_k(M)$, then $x \mathcal{I}_{k-1}(M) \subseteq \mathcal{I}_k(M) \subseteq \mathcal{I}_1(M) \cdot \mathcal{I}_{k-1}(M)$, where the latter inclusion follows by Lemma \ref{lem:generalinclusions}. Then, by \cite[Proposition 2.4]{Atiyah1969}, $x$ is integral over $\FR$. By the integral closure assumption, we conclude $x \in \FR$.
\end{proof}

In the literature, Pr\"{u}fer domains are defined in many equivalent ways. Below, we select some that are convenient for our purposes; see \cite[Theorem 1.1.1]{FHP1997} and \cite[Theorem 18.1 and Theorem 21.2]{gilmer1968} for more details.

\begin{definition}\cite{FHP1997, gilmer1968, HalterKoch}\label{def:prufer}
An integral domain $\FR$ is called a \emph{Pr\"{u}fer domain} if any of the following equivalent conditions hold:
\begin{enumerate}
\item All its nonzero finitely generated ideals are invertible;
\item All its nonzero ideals generated by two elements are invertible;
\item For every maximal ideal $\mathfrak{m} \subset \FR$, the localization of $\FR$ at $\mathfrak{m}$ is a valuation domain.
\end{enumerate}
\end{definition}

Examples of Pr\"{u}fer domains include every B\'{e}zout domain, every Dedekind domain, and the ring of integer-valued rational polynomials \cite{cahen1997}, i.e., elements of $\mathbb{Q}[x]$ that assume integer values for all $x \in \mathbb{Z}$. While maximal ideals are integral by definition, note that it does not really matter if the ideals in items 1-2 of Definition \ref{def:prufer} are taken to be integral or fractional, since nonzero principal ideals are always trivially invertible and $\langle r \rangle^{-1} = r^{-1} \FR$. As a preliminary step towards Theorem \ref{thm:smithchain}, we {recall} in Lemma \ref{lem:localglobal} \cite[Section 10]{Kaplansky1949} a property of ideals of a Pr\"{u}fer domain that will be also important to prove Theorem \ref{thm:rosenbrock}.

\begin{lemma}\label{lem:localglobal}\cite{Kaplansky1949}
   Let $\FR$ be a Pr\"{u}fer domain and let $\mathfrak{m} \subset \FR$ be a maximal ideal. Then, the localization $\FR_\mathfrak{m} = \left\{ \frac{a}{b} \mid a,b \in \FR, b \not\in \mathfrak{m} \right\}$ is an elementary divisor domain.
\end{lemma}
 \begin{proof}

  By item 3. in Definition \ref{def:prufer}, $\FR_\mathfrak{m}$ is a valuation domain, {and it is known that a valuation domain is an elementary divisor domain \cite[Section 10]{Kaplansky1949}}.
\end{proof}

\begin{theorem}\label{thm:smithchain}
Let $\FR$ be a Pr\"{u}fer domain and $M \in \FR^{m \times n}$. Then, the Smith ideals of $M$ are finitely generated. Moreover, for all $k=2,\dots,\min\{m,n\}$, it holds $\mathcal{S}_{k}(M) \subseteq \mathcal{S}_{k-1}(M)$.
\end{theorem}
\begin{proof}
If  $\mathcal{S}_k(M)=\langle 0 \rangle$ then the first part of the statement is obvious. Otherwise, $\mathcal{I}_{k-1}(M)$ is nonzero, and hence invertible because by definition $\mathcal{I}_h(M)$ are finitely generated for all $h$. Thus, there exist $a_i \in \mathcal{I}_{k-1}(M), x_i \in \mathcal{I}_{k-1}(M)^{-1}$ such that $1 = \sum_{i=1}^\ell a_i x_i$. It follows that, for every $y \in \mathcal{I}_{k-1}(M)^{-1}$ (so that $a_i y \in \FR$ for all $i$), $y=\sum_{i=1}^\ell (a_i y) x_i$, and hence $\mathcal{I}_{k-1}(M)^{-1} = \langle x_1, \dots, x_\ell \rangle$ is finitely generated. On the other hand, Proposition \ref{prop:invertible} implies $\mathcal{S}_k(M) = \mathcal{I}_k(M) \cdot \mathcal{I}_{k-1}(M)^{-1}$, and hence $\mathcal{S}_k(M)$ is generated by the products of the form $x_i m_j$, where $m_j$ is a $k \times k$ minor of $M$. There are finitely many, namely at most $\ell \binom{n}{k} \binom{m}{k}$, such products, and hence the Smith ideals are finitely generated.

To prove the second part of the statement, fix an arbitrary maximal ideal $\mathfrak{m} \subset \FR$, and denote by $I_\mathfrak{m} = \{ \frac{a}{b} \mid a \in I, b \in \FR, b \not\in \mathfrak{m} \}$ the localization at $\mathfrak{m}$ of the ideal $I$. Then, the localization $\mathcal{S}_k(M)_\mathfrak{m}$ is equal  to the $k$-th Smith ideal of $M$ seen as a matrix over $\FR_\mathfrak{m} \supseteq \FR$; see \cite[Proposition 3.11]{Atiyah1969} and \cite[item 4 in Theorem 3.4]{gilmer1968}. By Lemma \ref{lem:localglobal}, this implies in turn $\mathcal{S}_k(M)_\mathfrak{m} = \langle s_k \rangle$, where $s_k \in \FR_\mathfrak{m}$ is the $k$-th invariant factor in a Smith form of $M$ over $\FR_\mathfrak{m}$. We conclude that
\[ s_{k-1} \mid s_{k} \Rightarrow \mathcal{S}_k(M)_\mathfrak{m} \subseteq \mathcal{S}_{k-1}(M)_\mathfrak{m}.  \]
Since $\mathfrak{m}$ was arbitrary, we invoke the identity $I = \cap_\mathfrak{m} I_\mathfrak{m}$, which is valid for every ideal $I \subseteq \FR$ \cite[Proposition 3.9]{Atiyah1969}. This concludes the proof.
\end{proof}

\begin{example}\label{ex:noSmithdivisibility}
Let $\mathcal{R}=\mathbb{C}[x,y]$ and let $M=\begin{bmatrix}
x^3 y & 0 & x^4 &x y^3\\
0 & x y^3 & x^3 y & y^4
\end{bmatrix} \in \mathcal{R}^{2 \times 4}$. Then, $\mathcal{I}_1(M)=\mathcal{S}_1(M)=\langle x^4, x^3 y, x y^3, y^4 \rangle$, and $\mathcal{I}_2(M)=\langle x^4 y^4, x^6 y^2, x^3 y^5, x^5 y^3, x^2 y^6 \rangle$. Therefore, $x^2 y^2 \in \mathcal{S}_2(M)$, because $x^2 y^2 \mathcal{I}_1(M) = \langle x^6 y^2, x^5 y^3, x^3 y^5, x^2 y^6 \rangle \subsetneq \mathcal{I}_2(M)$; however, it is clear that $x^2 y^2 \not\in \mathcal{S}_1(M)$. It can in fact be shown, but we omit the derivation, that in this particular example the second Smith ideal is principal, $\mathcal{S}_2(M)=\langle x^2 y^2 \rangle$.
\end{example}

We now turn to generalizing the Smith-McMillan form. Given a matrix $M \in \K^{m \times n}$, let $0 \neq d \in \FR$ be such that $dM \in \mathcal{R}^{m \times n}$. Note that the existence of $d$ does \emph{not} require that $\FR$ is a GCD domain, e.g., we can just take $d$ to be the product of all the $mn$ denominators in $M$. We then define the determinantal fractional ideals of $M$ as $\mathcal{I}_k(M):=d^{-k} \mathcal{I}_k(dM)$, and the Smith ideals of $M$ as $\mathcal{S}_k(M):= d^{-1} \mathcal{S}_k(dM)$. It is clear that this definition is well posed, i.e., it does not depend on the choice of $d$. We record this construction in Definition \ref{def:smithidealsforrationalmatrices}.

\begin{definition}\label{def:smithidealsforrationalmatrices}
Let $\FR$ be an integral domain with field of fractions $\K$. Let $M \in \K^{m \times n}$ and ${0 \neq  } d \in \FR$ such that $dM \in \mathcal{R}^{m \times n}$. Then, for $k=1,\dots,\min\{m,n\}$, the $k$-th Smith ideal of $M$ is \[ \mathcal{S}_k(M):=d^{-1} \mathcal{S}_k(dM).\]
\end{definition}

The Smith ideals of a matrix over $\K$ are generally not integral ideals of $\FR$, not even if $\FR$ is a Euclidean domain (e.g. $\FR=\mathbb{Z}$ and $M=\frac12$). Over a Pr\"{u}fer domain, when $M$ is square and invertible over $\K$, Theorem \ref{thm:smithofinverse} shows that there is a nice relation between the Smith ideals of $M$ and those of $M^{-1}$.

\begin{theorem}\label{thm:smithofinverse}
Let $\FR$ be an integral domain and let $M \in \K^{n \times n}$ be invertible. Then, for all $k=0,\dots,n$,
\[  \mathcal{I}_k(M) =  \det(M)  \mathcal{I}_{n-k}(M^{-1}). \]
Moreover, if $\FR$ is a Pr\"{u}fer domain, then, for all $k=1,\dots,n$,
\[ \mathcal{S}_k(M)=\mathcal{S}_{n+1-k}(M^{-1})^{-1}. \]
\end{theorem}
\begin{proof}
The first part of the statement follows by the so-called Jacobi's identity \cite[Equation (0.8.4.1)]{HoJo}. For the second part of the statement, note that the invertibility of $M$ implies that the Smith ideals of $M$ and $M^{-1}$ must all be nonzero. Hence, using the first part of the statement,
\[  \mathcal{S}_k(M) = \mathcal{I}_{k}(M) \cdot \mathcal{I}_{k-1}(M)^{-1} = \mathcal{I}_{n-k}(M^{-1}) \cdot \mathcal{I}_{n+1-k}(M^{-1})^{-1} = \mathcal{S}_{n+1-k}(M^{-1})^{-1}.     \]
\end{proof}

Let us now define a sequence of ``numerator ideals" and ``denominator ideals". We start with denominators. 

\begin{definition}[Denominator and numerator ideals]\label{def:311}
Let $M \in \K^{m \times n}$ have Smith fractional ideals $\mathcal{S}_k(M)$. The $k$-th denominator ideal of $M$ is defined as follows: If $\mathcal{S}_k(M)=\langle 0 \rangle$, we define $\mathcal{D}_k(M):=\FR$; otherwise, we define $\mathcal{D}_k(M):= \mathcal{S}_k(M)^{-1} \cap \FR$. The $k$-th numerator ideal of $M$ is defined as $\mathcal{N}_k(M) := \mathcal{S}_k(M) \cdot \mathcal{D}_k(M)$.
\end{definition}

Note that the numerator and denominator ideals are always integral ideals of $\FR$. This is manifest by construction for $\mathcal{D}_k(M)$, whereas for the numerator ideal {it is immediate when
$\mathcal{S}_k(M)=\langle0\rangle$, and otherwise} it holds because
\[  \mathcal{S}_k(M) \cdot \mathcal{D}_k(M) \subseteq \mathcal{S}_k(M) \cdot \mathcal{S}_k(M)^{-1} \subseteq \FR.\]

\begin{proposition}\label{prop:equalifinv}
Let $\FR$ be an integral domain with field of fractions $\K$. For every $M \in \K^{m \times n}$ and all $k=1,\dots,\min\{m,n\}$, it holds $\mathcal{S}_k(M) \subseteq \mathcal{N}_k(M) : \mathcal{D}_k(M)$, with equality if $\mathcal{D}_k(M)$ is invertible.
\end{proposition}
\begin{proof}
If $\mathcal{S}_k(M) = \langle 0 \rangle$, the statement reduces to the identity $0=0$. Hence, let us assume that $\mathcal{S}_k(M)$ is nonzero. Let $s \in \mathcal{S}_k(M)$ and $d \in \mathcal{D}_k(M)$, then $sd \in \mathcal{N}_k(M)$ and hence $s \in \mathcal{N}_k(M) : \mathcal{D}_k(M)$. Assume now that $\mathcal{D}_k(M)$ is invertible, then by Proposition \ref{prop:invertible}
\[ \mathcal{S}_k(M) = \mathcal{S}_k(M) \cdot \mathcal{D}_k(M) \cdot \mathcal{D}_k(M)^{-1} =  \mathcal{N}_k(M) \cdot \mathcal{D}_k(M)^{-1}  = \mathcal{N}_k(M) : \mathcal{D}_k(M).\]  
\end{proof}
\begin{remark}
If $\FR$ is a Pr\"{u}fer domain then $\mathcal{S}_k(M)$ is finitely generated by Theorem \ref{thm:smithchain}. Thus, if $\mathcal{S}_k(M) \neq \langle 0 \rangle$, then $\mathcal{S}_k(M)^{-1}$ is finitely generated by the same argument as in the proof of Theorem \ref{thm:smithchain}. Hence, $\mathcal{D}_k(M) = \mathcal{S}_k(M)^{-1} \cap \FR = ( \mathcal{S}_k(M) + \FR)^{-1}$ is finitely generated and thus invertible by definition of Pr\"{u}fer domain. We conclude that over a Pr\"{u}fer domain it always holds $\mathcal{S}_k(M) = \mathcal{N}_k(M) \cdot \mathcal{D}_k(M)^{-1}$, and thus $\mathcal{N}_k(M) = \mathcal{S}_k(M) \cap \FR$.
\end{remark}

We show in Theorem \ref{thm:generalization} that the two chains of ideals $\mathcal{N}_k(M)$ and $\mathcal{D}_k(M)$ constitute indeed a generalization of the numerators and denominators of the invariant factors in a Smith-McMillan form.

\begin{theorem}\label{thm:generalization}
Let $\FR$ be an elementary divisor domain with field of fractions $\K$, and let $M \in \K^{m \times n}$ have rank $r$ and Smith-McMillan form $\bigoplus_{i=1}^r \frac{\epsilon_i}{\psi_i} \oplus 0$. Then:
\begin{enumerate}
\item For $k=1,\dots,r$, $\mathcal{N}_k(M)=\langle \epsilon_k \rangle$ and $\mathcal{D}_k(M)=\langle \psi_k \rangle$;
\item For $k > r$, $\mathcal{N}_k(M)=\langle 0 \rangle$ and $\mathcal{D}_k(M)=\FR$.
\end{enumerate}
\end{theorem}

\begin{proof}
By Theorem \ref{thm:unimodularinvariance}, we may with no loss of generality take $M$ to be equal to its Smith-McMillan form. 
\begin{enumerate}
\item Let first $k \leq r$. Since $\psi_k \mid \psi_1$ for all $k$, this implies that $\psi_1 M \in \FR^{m \times n}$ and therefore $\displaystyle \mathcal{I}_k(M) = \psi_1^{-k} \left\langle \prod_{i=1}^k \frac{\psi_1 \epsilon_i}{\psi_i} \right\rangle$ and $\displaystyle \mathcal{S}_k(M) = \psi_1^{-1} \left\langle \frac{\epsilon_k \psi_1}{ \psi_k} \right\rangle$. It follows that $\displaystyle \mathcal{S}_k(M)^{-1} = \epsilon_k^{-1} \langle \psi_k \rangle$ and hence $\mathcal{D}_k(M)= \langle \psi_k \rangle$. In turn, this yields $\mathcal{N}_k(M) = \langle \epsilon_k \rangle$.
\item In this  case $\mathcal{S}_k(M)=\langle 0 \rangle$ and hence $\mathcal{D}_k(M)=\FR$. Moreover, $\mathcal{N}_k(M) = \langle 0 \rangle \cdot \FR = \langle 0 \rangle$.
\end{enumerate}
\end{proof}

\section{The ideal-theoretic Rosenbrock's Theorem}\label{sec:rosenbrock}

In this section we often use block matrices, including special cases where some blocks may be empty. For example, if in the block matrix $\displaystyle X=\begin{bmatrix}
    X_{11} & 0\\
    0 & X_{22}
\end{bmatrix}$ we have $\displaystyle X_{11} = \begin{bmatrix}
    x & y\\
    z & w
\end{bmatrix} \in \FR^{2 \times 2}$ and $X_{22} \in \FR^{1 \times 0}$ (and the sizes of the zero blocks are determined accordingly), the notation is to be interpreted as to mean that one extra zero row must be appended, i.e., $\displaystyle X=\begin{bmatrix}
    x & y\\
    z & w\\
    0 & 0
\end{bmatrix}$. We will also identify, with slight abuse of notation, $\FR^{1 \times 1}$ with $\FR$.

\subsection{Rosenbrock's Theorem for matrices over Pr\"{u}fer domains}

Let $\FR$ be an integral domain and let $\K$ be its field of fractions. Consider a block system matrix
\begin{equation}\label{eq:systemmatrix}
P=\begin{bmatrix}
-A & B\\
C & D
\end{bmatrix} \in \FR^{(m+p) \times (n+p)},
\end{equation}
where $A \in \FR^{p \times p}$, $\det(A) \neq 0$. To state and prove an ideal-theoretic version of Rosenbrock's Theorem, we are going to relate the Smith ideals of $P$ to those of the Schur complement of $-A$ in $P$, namely, $G:=D+CA^{-1}B \in \K^{m \times n}$.

\begin{theorem}[Ideal-theoretic Rosenbrock's Theorem]\label{thm:rosenbrock}
Let $\FR$ be an integral domain with field of fractions $\K$. The following are equivalent:
\begin{enumerate}
\item $\FR$ is a Pr\"{u}fer domain;
\item For every block system matrix $P \in \FR^{(m+p) \times (n+p)}$ as in \eqref{eq:systemmatrix}, such that $A \in \FR^{p \times p}$ with $\det(A) \neq 0$, and such that $\mathcal{I}_p\left(\begin{bmatrix}
-A\\
C
\end{bmatrix}\right) = \mathcal{I}_p\left( \begin{bmatrix}
-A & B
\end{bmatrix} \right) = \FR$, defining $G= D + CA^{-1}B \in \K^{m \times n}$ then:
\begin{itemize}
\item[(i)] $\mathcal{S}_{p+1-k}(A) = \mathcal{D}_k(G)$ for all $k=1,\dots,\min\{m,n,p\}$.
\item[(ii)] $\mathcal{S}_k(A) = \FR$ for all $1 \leq k \leq p-\min\{m,n\}$
\item[(iii)] $\mathcal{D}_k(G)=\FR$ for all $p +1 \leq k \leq \min \{m,n\}$;
\item[(iv)] $\mathcal{S}_k(P)=\FR$ for all $k \leq p$;
\item[(v)] $\mathcal{S}_{k+p}(P) = \mathcal{N}_k(G)$ for all $k=1,\dots,\min\{m,n\}$.\end{itemize}
\end{enumerate}
\end{theorem}
\begin{proof}
We first show that the ideal-theoretic version of Rosenbrock's Theorem holds over any Pr\"{u}fer domain; then, we exhibit a counterexample valid over any non-Pr\"{u}fer integral domain.
\begin{itemize}
\item[$1 \Rightarrow 2$]

Fix a maximal ideal $\mathfrak{m} \subset \FR$. By Lemma \ref{lem:localglobal}, $\FR_\mathfrak{m} = \left\{ \frac{a}{b} \mid a,b \in \FR, b \not\in \mathfrak{m} \right\}$ is an elementary divisor domain. Hence, we can apply \cite[Theorem 3.5]{DNZ2025} to relate, over the ring $\FR_\mathfrak{m}$, the Smith forms of $P$ and $A$ and the Smith-McMillan form of $G$. Denoting by $I_\mathfrak{m}=\{ \frac{a}{b} \mid a \in I, b \in \FR, b \not\in \mathfrak{m} \}$ the localization of the ideal $I$ at $\mathfrak{m}$, it follows in turn by Theorem \ref{thm:generalization} that the equalities (i)-(v) hold for the localizations $\mathcal{S}_i(A)_\mathfrak{m}, \mathcal{S}_j(P)_\mathfrak{m}, \mathcal{D}_k(G)_\mathfrak{m}, \mathcal{N}_\ell(G)_\mathfrak{m}, \FR_\mathfrak{m}$. These identifications require the general identities in \cite[Theorem 3.4]{gilmer1968}. Namely, $I_\mathfrak{m} \cap J_\mathfrak{m} = (I \cap J)_\mathfrak{m}$ and, if $J$ is finitely generated, then  $(I:J)_\mathfrak{m} = I_\mathfrak{m} : J_\mathfrak{m}$. Since $\mathfrak{m}$ {was arbitrary}, the statement then follows from  the identity $I=\cap_\mathfrak{m} I_\mathfrak{m}$, which by \cite[Proposition 3.9]{Atiyah1969} holds for every ideal $I$ of $\FR$.

\item[$2 \Rightarrow 1$] Suppose that $\FR$ is not a Pr\"{u}fer domain. Then there exist two nonzero elements $r_1,r_2 \in \FR$ such that $\langle r_1, r_2 \rangle$ is not invertible. Take $A=\begin{bmatrix}
r_1 & 0\\
0 & r_2
\end{bmatrix}$, $B=C=I_2$, $D=0$, so that 

\begin{equation}\label{eq:mortimer}
P = \begin{bmatrix}
-r_1 & 0 & 1 & 0\\
0 & -r_2 & 0 & 1\\
1 & 0 & 0 & 0\\
0& 1 & 0 & 0
\end{bmatrix} , \qquad G = \begin{bmatrix}
\frac{1}{r_1} & 0\\
0 & \frac{1}{r_2}
\end{bmatrix}.
\end{equation}
Manifestly, $P$ is unimodular and so $\mathcal{S}_k(P)=\FR$ for all $k$. On the other hand, $\mathcal{S}_1(G)=\frac{1}{r_1r_2} \langle r_1,r_2 \rangle$. Thus, by Lemma \ref{lem:inverse2gen}, $\mathcal{S}_1(G)^{-1} = \langle r_1 \rangle \cap \langle r_2 \rangle = \mathcal{D}_1(G)$. Using again Lemma \ref{lem:inverse2gen}, and since by assumption $\langle r_1, r_2 \rangle$ is not invertible, \[ \mathcal{N}_1(G) = \mathcal{S}_1(G) \cdot \mathcal{D}_1(G) = \langle r_1, r_2 \rangle \cdot \langle r_1, r_2 \rangle^{-1} \neq \FR = \mathcal{S}_3(P). \]
\end{itemize}
\end{proof}

\subsection{Hierarchy of coprimality notions beyond elementary divisor domains}

In \cite{DNZ2025}, Rosenbrock's Theorem is stated over elementary divisor domains under the following equivalent  \cite[Proposition 2.1]{DNZ2025} assumptions on the block matrices $\begin{bmatrix}
-A & B
\end{bmatrix}$ and $\begin{bmatrix}
-A\\
C
\end{bmatrix}$: 
\begin{itemize}
\item[(a)] their Smith form is trivial;
\item[(b)]they are, resp., right invertible and left invertible over $\FR$;
\item[(c)]they can be brought to Smith form just by, resp., right and left multiplication by a unimodular matrix;
\item[(d)]the pairs $A,B$ and $A,C$ are, resp., left and right coprime;
\item[(e)]they are completable to a unimodular matrix.
\end{itemize}
For an elementary divisor domain, it is clear via Theorem \ref{thm:generalization} that these assumptions are all equivalent to the ideal-theoretic condition 
\begin{itemize}
\item[(f)] $\mathcal{I}_p\left(\begin{bmatrix}
-A\\
C
\end{bmatrix}\right) = \mathcal{I}_p\left( \begin{bmatrix}
-A & B
\end{bmatrix} \right) = \FR$.
\end{itemize}  (We give an explicit proof of this claim in Proposition \ref{prop21} below). However, the situation is not the same for a general integral domain. Indeed, clearly conditions (a) and (c) are not applicable in general because a Smith form may not exist. Moreover, Proposition \ref{prop21} discusses the chain of implications of conditions (b), (d), (e), (f); we omit a discussion of the, obviously identical, chain of implications concerning the matrix $\begin{bmatrix}
-A\\
C
\end{bmatrix}$.

\begin{proposition}\label{prop21}
Let $\mathcal{R}$ be an integral domain and let $\begin{bmatrix}
-A & B
\end{bmatrix} \in \mathcal{R}^{p \times (p+n)}$ where $A \in \mathcal{R}^{p \times p}$. Consider the following properties:
\begin{enumerate}
\item $\mathcal{I}_p(\begin{bmatrix}
-A & B
\end{bmatrix}) = \mathcal{R}$;
\item $\begin{bmatrix}
-A & B
\end{bmatrix}$ is right invertible;
\item $A$ and $B$ are left coprime, i.e., if $U \in \FR^{p \times p}$ satisfies $UZ=\begin{bmatrix}
-A & B
\end{bmatrix}$ for some $Z \in \FR^{p \times (p+n)}$, then $\det(U) \in \FR^\times$;
\item There exist $C,D$ of suitable size such that $\begin{bmatrix}
-A & B\\
C & D
\end{bmatrix} \in \FR^{(p+n) \times (p+n)}$ is unimodular.
\end{enumerate}
Then, $4 \Rightarrow 1 \Leftrightarrow 2 \Rightarrow 3$, but generally the reverse implications do not hold.
\end{proposition}
\begin{proof}
We first show that $1.$ and $2.$ are equivalent.
\begin{itemize}
\item[$1 \Rightarrow 2$] Let $N:=\binom{n+p}{p}$. For $i=1,\dots,N$, index by $M_i$ all the $p \times p$ submatrices of $\begin{bmatrix}
-A & B
\end{bmatrix}$ (having fixed the order of their columns as in the original matrix). By construction, there exists a permutation matrix $P_i \in \mathcal{R}^{(p+n) \times (p+n)}$ such that $M_i = \begin{bmatrix}
-A & B
\end{bmatrix} P_i \begin{bmatrix}
I_p\\
0
\end{bmatrix}$. On the other hand, $M_i \mathrm{adj}(M_i)=\det(M_i) I_p$ and by assumption there exist $c_i \in \mathcal{R}$ such that \[ 1 = \sum_{i=1}^{N}c_i \det(M_i) \Rightarrow \begin{bmatrix}
-A & B
\end{bmatrix} \left( \sum_{i=1}^N c_i P_i \begin{bmatrix}
\mathrm{adj}(M_i)\\
0
\end{bmatrix}   \right) = I_p. \]
\item[$2 \Rightarrow 1$] Assume that $I_p=-AX+BY$ for appropriate matrices $ X,Y$ over $\mathcal{R}$. Taking determinants, it follows immediately by the Cauchy-Binet formula that $1$ is expressible as a linear combination of the $p \times p$ minors of $\begin{bmatrix}
-A & B
\end{bmatrix}$.
\end{itemize}
We now discuss the other implications or lack thereof.
\begin{itemize}
\item[$1 \Rightarrow 3$] By the Cauchy-Binet formula, $\FR = \mathcal{I}_p\left(\begin{bmatrix}
-A & B
\end{bmatrix}\right) \subseteq \langle \det(U) \rangle \subseteq \FR$.
\item[$3 \not\Rightarrow 1$] Consider for example $\FR=\mathbb{C}[x_1,x_2]$ and take $A=x_1$, $B=x_2$.
\item[$4 \Rightarrow 2$] A right inverse is given by the appropriate block column of $\begin{bmatrix}
-A & B\\
C & D
\end{bmatrix}^{-1}$.
\item[$2 \not \Rightarrow 4$] \cite[p. 270]{Swan1962} Let $\FR=\mathbb{R}[x_1,x_2,x_3]/\langle x_1^2+x_2^2+x_3^2-1 \rangle$ and take $A=-x_1$, $B=\begin{bmatrix}
x_2 & x_3
\end{bmatrix}$. Then $\begin{bmatrix}
x_1\\
x_2\\
x_3
\end{bmatrix}$ is manifestly a right inverse, but a unimodular completion cannot be found for topological reasons \cite{Swan1962}.
\end{itemize}
\end{proof}

Combining the second half of the proof of Theorem \ref{thm:rosenbrock} with Proposition \ref{prop21}, we see that Rosenbrock's Theorem cannot hold for non-Pr\"{u}fer domains not even assuming condition $4.$ of Proposition \ref{prop21}, because the counterexample \eqref{eq:mortimer}  satisfies even this stronger assumption. However, the first part of the proof of Theorem \ref{thm:rosenbrock} for Pr\"{u}fer domains may leave the reader in doubt about whether, in that class of rings, the weaker coprimality condition $3.$ of Proposition \ref{prop21} may suffice to establish the theorem. The answer is negative. Note that the proof itself would break down by weakening the assumption in this manner, because (unlike the ideal-theoretic condition $1.$ of Proposition \ref{prop21}) coprimality is not generally preserved by localization. Indeed, the statement itself is not always true under the weaker assumption $3.$ We exhibit a counterexample in Example \ref{ex:final}.

\begin{example}\label{ex:final}
Consider $\FR = \mathbb{Z}[\sqrt{-5}]$, which is a Dedekind domain (and hence a Pr\"{u}fer domain) but not a GCD domain \cite[Chapter 3]{Marcus1977}. Take $A=2$, $B=1+\sqrt{-5}$, $C=1-\sqrt{-5}$, $D=0$. Then,
\begin{equation}\label{eq:mortimer2}
P  = \begin{bmatrix}
-2 & 1+\sqrt{-5}\\
1-\sqrt{-5} & 0
\end{bmatrix} \in \FR^{2 \times 2}, \qquad G = 3 \in \mathbb{Q}(\sqrt{-5}). 
\end{equation}
The block matrix $P$ in \eqref{eq:mortimer2} does not satisfy the assumptions of the ideal-theoretic Rosenbrock's Theorem \ref{thm:rosenbrock} because $\langle 2, 1 + \sqrt{-5} \rangle \subsetneq \FR$. Nonetheless, it is a very famous example in number theory that the pairs $2, 1 \pm \sqrt{-5}$ do not have a non-unit common divisor in $\FR$ \cite[pp. 42-43]{Marcus1977}, and therefore $P$ satisfies the weaker coprimality assumption.

Let us now demonstrate that the predictions of Rosenbrock's Theorem are incorrect for the example \eqref{eq:mortimer2}. Obviously, $\mathcal{S}_1(A)=\langle 2 \rangle $ and  $\mathcal{S}_1(G)=\mathcal{N}_1(G)=\langle 3 \rangle$, implying $\mathcal{D}_1(G)=\FR \neq \mathcal{S}_1(A)$. It is only slightly more laborious to verify that $\mathcal{S}_1(P) = \langle 2, 1 + \sqrt{-5} \rangle \neq \FR$ and $\mathcal{S}_2(P)=3 \mathcal{S}_1(P) \neq \mathcal{N}_1(G)$.
\end{example}

\subsection{Some special cases of validity beyond Pr\"{u}fer domains}

We now prove in Theorem \ref{prop:vectorcase} that, if either $m=1$ or $n=1$ (or both), the ideal-theoretic Rosenbrock's Theorem holds even when $\FR$ is not a Pr\"{u}fer domain. This includes in particular the case of a scalar Schur complement $G \in \K$.

\begin{theorem}\label{prop:vectorcase}
    Let $\FR$ be an integral domain with field of fractions $\K$, and suppose that the matrices $P,A,B,C,D,G$ are defined as in, and satisfy the assumptions of, item 2. of Theorem \ref{thm:rosenbrock}. If $\min \{m,n\}=1$, then the equalities (i)-(v) of Theorem \ref{thm:rosenbrock} hold.
\end{theorem}
\begin{proof}
    We give a proof assuming $n=1$; the case $m=1$ is analogous and omitted. To alleviate the notation, throughout the proof we denote $H:=\det(A) G \in \FR^m$.
    \begin{itemize}
        \item[(ii)] Up to a sign, the $p \times p$ minors of $\begin{bmatrix}
        -A & B
    \end{bmatrix}$ are precisely $\det(A)$ and, by Cramer's rule, the elements of the vector $v:=\mathrm{adj}(A)B \in \FR^{p}$. Hence, using Lemma \ref{lem:generalinclusions}, \[ \FR = \mathcal{I}_p \left( \begin{bmatrix}
        -A  & B
    \end{bmatrix} \right) = \langle \det(A),v_1,\dots,v_p \rangle \subseteq \mathcal{I}_{p-1}(A) \subseteq \cdots \subseteq \mathcal{I}_1(A) \subseteq \FR . \]
    \item[(iii)] Vacuously true. 
    \item[(iv)] Using also Lemma \ref{lem:generalinclusions},
    \[ \FR = \mathcal{I}_p\left( \begin{bmatrix}
        -A & B
    \end{bmatrix} \right) \subseteq \mathcal{I}_p(P) \subseteq \cdots \subseteq \mathcal{I}_1(P) \subseteq \FR.   \]
    \item[(i)] By item (ii), $\mathcal{S}_p(A)= \langle \det(A) \rangle$. Observe that $\displaystyle \mathcal{S}_1(G)=(\det(A))^{-1} \mathcal{I}_1(H)$, implying in turn \[ \mathcal{D}_1(G) = \{ a \in \FR \mid a \mathcal{S}_1(G) \subseteq \FR \} = \{ a \in \FR \mid a \mathcal{I}_1(H) \subseteq \langle \det(A) \rangle  \}.   \]
    Hence, it is obvious that $\langle \det(A) \rangle \subseteq \mathcal{D}_1(G)$; we must prove the reverse inclusion.
    Defining $v \in \FR^p$ as in the proof of item (ii), we have $\displaystyle H = \det(A) D + C v$. We claim that $\langle \det(A) \rangle +  \mathcal{I}_1(H) = \FR$, or equivalently there exist $x \in \FR$ and $y \in \mathcal{I}_1(H)$ such that $x \det(A)+y=1$. Suppose $a \in \mathcal{D}_1(G)$, then $ay=b \det(A)$ for some $b \in \FR$, and hence $a = a x \det(A) + a y = \det(A) (ax+b) \in \langle \det(A) \rangle$.

    It remains to prove the claim. By the proof of item (ii) and by Proposition \ref{prop21}, there exist a scalar $x_1 \in \FR$, a vector $w_1 \in \FR^p$, and two matrices $Z_1 \in \FR^{p \times p}, Z_2 \in \FR^{p \times m}$ such that $x_1 \det(A) + w_1^T \mathrm{adj}(A) B = 1 $ and $Z_2 C - Z_1 A = I_p$. Define $w:=Z_2^T w_1 \in \FR^m$, $y:=w^T H = \det(A) w_1^T Z_2 G \in \mathcal{I}_1(H)$ and $x:= x_1 - w_1^T(Z_1 B + Z_2 D) \in \FR$. Then, 
    \[ x \det(A) + y = x_1 \det(A) + w_1^T ( Z_2 C \mathrm{adj}(A) B - Z_1 \det(A) B ) = x_1 \det(A) + w_1^T \mathrm{adj}(A) B =1   , \]
    where in the last step we used the identity $\det(A) I_p = A \cdot \mathrm{adj}(A)$.

    \item[(v)] By item (iv), it suffices to prove $\mathcal{I}_{p+1}(P)=\mathcal{N}_1(G)$. Using \cite[Theorem 4.2]{DNZ2025}, we have
    $ \displaystyle \det\left(\begin{bmatrix}
        -A & B\\
        e_\ell^T C & e_\ell^T D
    \end{bmatrix}\right) = \det(-A) e_\ell^T G  $ for all $\ell=1,\dots,m$. Therefore,
    \[ \mathcal{N}_1(G) = \mathcal{S}_1(G)  \cdot \mathcal{D}_1(G) = \mathcal{I}_1(H) \subseteq \mathcal{I}_{p+1}(P). \]
    For the reverse inclusion, observe that an arbitrary $(p+1) \times (p+1)$ minor of $P$ has the form $\det(EP)$, where $E \in \{0,1\}^{(p+1) \times (p+m)}$ is a matrix that selects $(p+1)$ rows. On the other hand, $\displaystyle E P \begin{bmatrix}
        \mathrm{adj}(A) B\\
        \det(A)
    \end{bmatrix} = E \begin{bmatrix}
        0\\
        H
    \end{bmatrix}.$ Therefore, by Cramer's rule and letting $M_i$ denote the matrix obtained from $EP$ by replacing the $i$-th column with $\displaystyle E\begin{bmatrix}
        0\\
        H
    \end{bmatrix}$, we have $\displaystyle \det(EP) \left( \begin{bmatrix}
        \mathrm{adj}(A)B\\
        \det(A)
        \end{bmatrix}\right)_i = \det(M_i) \in \mathcal{I}_1(H)$, for $i=1,\dots,p+1$.
    Taking into account that $\mathcal{I}_1\left( \begin{bmatrix}
        \mathrm{adj}(A)B\\
        \det(A)
        \end{bmatrix}\right)=\mathcal{I}_p \left( \begin{bmatrix}
        -A & B
    \end{bmatrix} \right)=\FR $, we conclude $\mathcal{I}_{p+1}(P) \subseteq \mathcal{I}_1(H)$.

    \end{itemize}

\end{proof}

Theorem \ref{prop:easy} is an analogue to Theorem \ref{prop:vectorcase}, but under the assumption $p=1$ and for integrally closed domains. Before we state and prove it, we give the technical Lemma \ref{lem:claim}.

\begin{lemma}\label{lem:claim}
    Let $\FR$ be an integral domain with field of fractions $\K$, and suppose that the matrices $P,A,B,C,D,G$ are defined as in, and satisfy the assumptions of, item 2. of Theorem \ref{thm:rosenbrock}. If $p=1$, then $\mathcal{I}_{k+1}(P)=A \mathcal{I}_k(G)$ for all $k=1,\dots,\min \{m,n\}$.
\end{lemma}

\begin{proof}
    We claim that the statement holds locally at an arbitrary maximal ideal $\mathfrak{m}$; invoking \cite[Proposition 3.9]{Atiyah1969} then concludes the proof. To prove the claim, there are two cases.
        \begin{itemize}
            \item If $A \not\in \mathfrak{m}$, then $A$ is invertible in $\FR_\mathfrak{m}$, and therefore $P$ is unimodularly (over $\FR_\mathfrak{m}$) equivalent to $\displaystyle \begin{bmatrix}
            1 & 0\\
            0 & G
        \end{bmatrix}$, yielding the claim.
        \item If $A \in \mathfrak{m}$, then $\FR_\mathfrak{m}=\langle A \rangle_\mathfrak{m} + \mathcal{I}_1(B)_\mathfrak{m} =\mathcal{I}_1(B)_\mathfrak{m}$. Since $\FR_\mathfrak{m}$ is a local ring, this implies that the vector $B$ contains a unit, and hence we can construct a unimodular (over $\FR_\mathfrak{m}$) matrix $V$ such that $B  V= e_1^T$. (For example, if $\displaystyle B=\begin{bmatrix}
            \beta & B_1
        \end{bmatrix}$ where $\beta$ is a unit, which is no loss of generality up to right multiplication by a permutation matrix, then $V=\displaystyle \begin{bmatrix}
            \beta^{-1} & -\beta^{-1} B_1\\
            0 & I
        \end{bmatrix} $ is such a unimodular matrix.) Similarly, there exists a unimodular (over $\FR_\mathfrak{m}$) matrix $U$ such that $U C = e_1$. Hence, 
        \[  \begin{bmatrix}
            1 & 0\\
            0 & U
        \end{bmatrix} P \begin{bmatrix}
            1 & 0\\
            0 & V
        \end{bmatrix} = \begin{bmatrix}
            -A & e_1^T \\
            e_1 & U D V
        \end{bmatrix} =: \begin{bmatrix}
            -A & 1 & 0\\
            1 & D_{11} & D_{12}\\
            0 & D_{21} & D_{22}
        \end{bmatrix}, \]
        where the last step is just one more block partition, with $D_{11} \in \FR_{\mathfrak{m}}$ and $D_{22} \in \FR_{\mathfrak{m}}^{(m-1) \times (n-1)}$. By performing further elementary unimodular equivalences, and noting that $\alpha :=  1 + A D_{11} \equiv 1 \bmod \mathfrak{m}$ is a unit of $\FR_\mathfrak{m}$, we conclude that $P$ is unimodularly equivalent to
        \[   \begin{bmatrix}
            0 & 1 & 0\\
            \alpha & D_{11} & D_{12}\\
            0 & \alpha^{-1} D_{21} & F
        \end{bmatrix},  \qquad F : = D_{22} - \alpha^{-1} A D_{21} D_{12}. \]
        Just a couple more elementary row and column transformations lead us to conclude that $P$ and $\displaystyle \begin{bmatrix}
            I_2 & 0\\
            0 & F
        \end{bmatrix}$ are unimodularly equivalent over $\FR_\mathfrak{m}$.
        This in turn implies $\mathcal{I}_{k+1}(P)_\mathfrak{m} = \mathcal{I}_{k-1}(F)_\mathfrak{m}$ by Theorem \ref{thm:unimodularinvariance}.  On the other hand, \[  U (A G) V =A (U D V) + e_1 e_1^T = \begin{bmatrix}
            \alpha & A D_{12}\\
            A D_{21} & A D_{22}
        \end{bmatrix}, \]
        and thus $AG$ is unimodularly equivalent to $\displaystyle \begin{bmatrix}
            \alpha & 0\\
            0  & A F
        \end{bmatrix}$ thus proving via Theorem \ref{thm:unimodularinvariance} that $A^k \mathcal{I}_k(G)_\mathfrak{m}=\mathcal{I}_k(AG)_\mathfrak{m} = \mathcal{I}_{k-1}(AF)_\mathfrak{m} = A^{k-1} \mathcal{I}_{k-1}(F)_\mathfrak{m}$, leading to $A \mathcal{I}_k(G)_\mathfrak{m} = \mathcal{I}_{k-1}(F)_\mathfrak{m} = \mathcal{I}_{k+1}(P)_\mathfrak{m}$. This shows the claim.
        \end{itemize}
\end{proof}

\begin{theorem}\label{prop:easy}
    Let $\FR$ be an integrally closed domain with field of fractions $\K$, and suppose that the matrices $P,A,B,C,D,G$ are defined as in, and satisfy the assumptions of, item 2. of Theorem \ref{thm:rosenbrock}. If $p=1$, then the equalities (i)-(v) of Theorem \ref{thm:rosenbrock} hold.
\end{theorem}

\begin{proof}
Once again, we denote $H:=\det(A) G = A G = AD+CB$.
    \begin{itemize}
        \item[(i)] The inclusion $\langle A \rangle \subseteq \mathcal{D}_1(G)$ is obvious. For the reverse inclusion, the assumptions yield
        \[ \FR = \FR \cdot \FR = ( \langle A \rangle + \mathcal{I}_1(B)) \cdot ( \langle A \rangle + \mathcal{I}_1(C)) \subseteq \langle A \rangle + \mathcal{I}_1(CB) = \langle A \rangle + \mathcal{I}_1(H) \subseteq \FR.  \]
        Hence, there exist $x \in \mathcal{I}_1(H), y \in \FR$ such that $1=x+yA$. If $z \in \mathcal{D}_1(G)$, then $zx \in \langle A \rangle$ and thus $z=zx+zyA \in \langle A \rangle$.
        \item[(ii)] Vacuously true.
        \item[(iii)] By Lemma \ref{lem:claim}, for $k \geq 2$ it holds $\mathcal{S}_k(G) = \mathcal{S}_{k+1}(P) \subseteq \FR$, where the last inclusion follows by Theorem \ref{thm:smithintegral}. Hence, {$\mathcal{D}_k(G)=\FR$ by Definition \ref{def:311}.} 
        \item[(iv)] Taking into account that $p=1$, the proof is the same as item (iv) in Theorem \ref{prop:vectorcase}.
        \item[(v)] The statement follows by Lemma \ref{lem:claim}, because $\mathcal{N}_1(G) = A\mathcal{S}_1(G)=A \mathcal{I}_1(G)$ and, for $k \geq 2$, $\mathcal{N}_k(G) = \mathcal{S}_k(G) = \mathcal{S}_{k+1}(P)$. 
    \end{itemize}
\end{proof}

\subsection{The case of polynomial rings}

Proposition \ref{prop21} can be seen as an  algebraic abstraction, valid for more general rings, of classic observations that have been made, in the context of systems theory, for matrices over $\FR=\mathbb{C}[x_1,\dots,x_d]$ or $\FR=\mathbb{R}[x_1,\dots,x_d]$. For instance, in this specifically restricted setting, Zerz \cite{Zerz1996} already established the hierarchy between condition 1. (called Zero Primeness in \cite{Zerz1996}) and condition 3. (called {Factor} Primeness in \cite{Zerz1996}) in Proposition \ref{prop21}. Still in the context of control theory \cite{AGP2025} and over the ring of polynomials with complex coefficients, the prerequisites of Rosenbrock's Theorem are often expressed by requiring full rank of $\begin{bmatrix}
-A(x_1,\dots,x_d) & B(x_1,\dots,x_d)
\end{bmatrix}$ for all possible values of $(x_1,\dots,x_d) \in \mathbb{C}^d$ (and similarly for the block involving $-A$ and $C$); via Hilbert's Nullstellensatz, such requirement is indeed equivalent to $\mathcal{I}_p\left( \begin{bmatrix}
-A & B
\end{bmatrix} \right)=\C[x_1,\dots,x_d]$. Again in the context of polynomial rings, completability to a unimodular matrix (condition 4. in Proposition \ref{prop21}) is also equivalent to the same requirements, by the Quillen-Suslin Theorem \cite{Quillen1976}.

We can, in fact, say something more about the ideal-theoretic Rosenbrock's Theorem in the case of polynomial rings with coefficients in an algebraically closed field $\mathbb{F}$. Obviously, item 2. in Theorem \ref{thm:rosenbrock} always holds over univariate polynomials. In the multivariate case, there exist matrices for which the theorem fails; to see this, just specialize \eqref{eq:mortimer} to $r_1=x_1,r_2=x_2$ when $\FR=\mathbb{F}[x_1,\dots,x_d]$ and $d \geq 2$. Nonetheless, Proposition \ref{rem:generic} shows that, if $d \leq 3$, Theorem \ref{thm:rosenbrock} holds for \emph{almost every} square\footnote{At the price of making the proofs somewhat more technical, one could analyze rectangular matrices in an analogous manner, but we prefer to focus on the square case in order to keep this part of the paper concise.} matrix in the bivariate and trivariate cases. This fact provides some evidence that the ideal-theoretic Rosenbrock's Theorem may still be a valuable tool for applications in control theory. For some recent results on the related problem of which matrices over multivariate polynomial rings are unimodularly equivalent to a matrix in Smith form, we recommend that interested readers may consult \cite{LLW2024,LWXZ2026} and the references therein.

\begin{lemma}\label{lem:generic}
Let $\mathbb{F}$ be an algebraically closed field and $\FR=\mathbb{F}[x_1,\dots,x_d]$ where $d \leq 3$. {Fix a positive integer $p$ and a nonnegative integer $\ell$, and let $N:=p^2\binom{d+\ell}{d}$. Identify the matrices $M \in \FR^{p \times p}$ having entries of total
degree at most $\ell$ with the vectors
in $\mathbb{F}^N$ containing the coefficients of their entries. Then, there exists a nonempty Zariski-open subset
$U\subseteq\mathbb{F}^N$ such that, for every $M$ with coefficient vector in $U$, the determinantal} ideals of $M$ are {$\mathcal{I}_k(M) = \FR$} for $k<p$ and {$\mathcal{I}_p(M) = \langle \det(M) \rangle$}.
\end{lemma}
\begin{proof}
{If $p=1$, we take $U=\mathbb{F}^N$; otherwise,} by the block factorization
\[  \begin{bmatrix}
X_{11} & X_{12}\\
X_{21} & X_{22}
\end{bmatrix} = \begin{bmatrix}
I & 0\\
X_{21}X_{11}^{-1} & I
\end{bmatrix}   \begin{bmatrix}
X_{11} & 0\\
0 & X_{22}-X_{21}X_{11}^{-1}X_{12}
\end{bmatrix}    \begin{bmatrix}
I & X_{11}^{-1} X_{12}\\
0 & I
\end{bmatrix},\]
we see that (up to permuting rows and columns so that the principal leading submatrix is invertible) a matrix $X \in \mathbb{F}^{p \times p}$ has nullity $q$ if and only if its elements satisfy one polynomial inequality $\det X_{11} \neq 0$ and $q^2$ rational equalities $X_{22}=X_{21}X_{11}^{-1}X_{12}$. 

Suppose $\mathcal{I}_{p-1}(M) \neq \FR$. Then, by Hilbert's Nullstellensatz, there exists $y \in \mathbb{F}^d$ that is simultaneously a zero of every $(p-1) \times (p-1)$ minor of $M$. Letting $X=M(y) \in \mathbb{F}^{p \times p}$, this implies that the nullity of $X$ is $q \geq 2$, and therefore the matrix elements of $X$ satisfy $q^2 \geq 4$ rational equations. {For each fixed $y$, these equations impose $q^2$ independent
conditions on the coefficients of $M$, and allowing $y$
to vary adds $d \leq 3$ degrees of freedom. Hence, for each of the finitely many choices
of $q$ and of the row and column permutations, the
exceptional coefficient vectors have Zariski closure
of dimension at most $N+d-q^2<N$. Taking the finite
union over these choices constructs a proper Zariski-closed subset
$\mathfrak{Z}\subsetneq\mathbb{F}^N$, and we take $U=\F^N \setminus \mathfrak{Z}$.}
\end{proof}

\begin{proposition}\label{rem:generic}
Let $\mathbb{F}$ be an algebraically closed field and $\FR=\mathbb{F}[x_1,\dots,x_d]$ where $d \leq 3$. {Fix positive integers $p,n$ and a nonnegative integer $\ell$, and let $N:=(p+n)^2\binom{d+\ell}{d}$. Identify the block system matrices $P$ in
\eqref{eq:systemmatrix}, with $m=n$ and entries of total
degree at most $\ell$, with the vectors
in $\mathbb{F}^N$ containing the coefficients of their entries. Then, there exists a nonempty Zariski-open subset
$U\subseteq\mathbb{F}^N$ such that, for every $P$ with
coefficient vector in $U$, if $P$ satisfies the assumptions
of item 2. of Theorem \ref{thm:rosenbrock}, then the
equalities (i)--(v) of Theorem \ref{thm:rosenbrock} hold. }
\end{proposition}
\begin{proof}
{The case $n=1$ follows from Theorem \ref{prop:vectorcase}, and hence we may assume $n\geq 2$.} Applying Lemma \ref{lem:generic} to $A$ and $P$ we get that, Zariski-generically in the coefficients of the entries of each independent block submatrix $A,B,C,D$, {i.e., for all block system matrices $P$ corresponding to points in a nonempty Zariski-open subset $U \subseteq \F^N$,} 

\begin{equation}\label{eq:genericdeterminantideals}
{\mathcal{I}_k(P) = \begin{cases}
\FR \ &\mathrm{if} \ k < n+p;\\
\langle \det(P) \rangle \ &\mathrm{if} \ k=n+p,
\end{cases}  \quad \mathcal{I}_k(A) = \begin{cases}
\FR \ &\mathrm{if} \ k < p;\\
\langle \det(A) \rangle \ &\mathrm{if} \ k=p.
\end{cases} }   
\end{equation}
{Fix now a system matrix $P$ corresponding to a point in $U$, and assume further that $\det(A) \neq 0$ and $\mathcal{I}_p\left(\begin{bmatrix}
    -A&B
\end{bmatrix}\right)=\mathcal{I}_p\left(\begin{bmatrix}
    -A\\
    C
\end{bmatrix}\right)=\FR$. For $k=1,\dots,n$, denote by $\mathcal{J}_{p+k}(P) \subseteq \mathcal{I}_{p+k}(P)$ the ideal generated by the $(p+k) \times (p+k)$ minors of $P$ bordering $-A$, i.e., determinants of submatrices of $P$ including $-A$ as a sub-submatrix. Translating \cite[Theorem 4.2]{DNZ2025} into the language of ideals, we have $\langle \det(A) \rangle \cdot \mathcal{I}_k(G) = \mathcal{J}_{p+k}(P)$. Clearly $\mathcal{J}_{p+n}(P) = \mathcal{I}_{p+n}(P) = \langle \det(P) \rangle$, and we claim that $\mathcal{J}_{p+k}(P)=\FR$ for $k < n$.} Then, the statement follows from the calculation
\[  \mathcal{S}_k(G) = \mathcal{I}_k(G) : \mathcal{I}_{k-1}(G) = \begin{cases}
\det(A)^{-1} \FR \ &\mathrm{if} \ k=1;\\
\FR \ &\mathrm{if} \ 1 < k < n;\\
\langle \det(P) \rangle \ &\mathrm{if} \  k=n.
\end{cases} \]
It remains to prove the claim, {which by Hilbert's Nullstellensatz is in turn equivalent to the following statement: For every $y \in \F^d$ and every $k =1,\dots,n-1$,  there exists a nonzero $(p+k) \times (p+k)$ minor of $P(y)$ bordering $-A(y)$. To this goal, fix $y \in \F^d$ and note that \eqref{eq:genericdeterminantideals} implies that $\mathrm{rank} \ A(y) \geq p-1$. If $A(y)$ is invertible, then $G(y)$ is defined and, combining \cite[Remark 3.4]{DNZ2025} and \eqref{eq:genericdeterminantideals}, $\mathrm{rank} \  G(y) = \mathrm{rank} \ P(y)-p \geq n-1$, implying that there is a nonzero $k \times k$ minor of $G(y)$ for all $1 \leq k < n$. By \cite[Theorem 4.2]{DNZ2025}, this in turn implies that, for all $1 \leq k < n$, there is a nonzero $(p+k) \times (p+k)$ minor of $P(y)$ bordering $-A(y)$. If instead $-A(y)$ has rank $p-1$, then there exist invertible $X,Z \in \F^{p \times p}$ such that $XA(y)Z=I_{p-1} \oplus 0$. Furthermore, $e_p^T X B(y) \neq 0$, because otherwise $e_p^T X\begin{bmatrix}
 -A(y) & B(y)   
\end{bmatrix}=0$ contradicting $\mathcal{I}_p\left(\begin{bmatrix}
    -A & B
\end{bmatrix}\right)=\FR$,  and similarly $C(y) Z e_p \neq 0$ because otherwise $\begin{bmatrix}
 -A(y) \\
 C(y)
\end{bmatrix} Z e_p=0$ which is not compatible with the assumption $\mathcal{I}_p\left(\begin{bmatrix}
    -A \\
    C
\end{bmatrix}\right)=\FR$ .  Let $Q:=(X \oplus I_n) P(y) (Z \oplus I_n)$, and observe that the minors of $Q$ bordering $-I_{p-1} \oplus 0$ are equal to the corresponding minors of $P(y)$ bordering $-A(y)$ times $\det(X)\det(Z) \neq 0$; therefore, it suffices to find some appropriate nonzero minors of $Q$ bordering $-I_{p-1} \oplus 0$. One submatrix of $Q$ including $-I_{p-1} \oplus 0$ as a sub-submatrix is
\[   Q_0=\begin{bmatrix}
    -I_{p-1} & 0 & \star\\
    0 & 0 & b\\
    \star & c & \star
\end{bmatrix} \in \F^{(p+1) \times (p+1)},   \]
where $b$ (resp. $c$) is a nonzero element of $e_p^T X B(y)$ (resp. of $C(y)Ze_p$) and the symbol $\star$ denotes sub-submatrices whose exact nature is irrelevant. A simple calculation yields $\det (Q_0) = (-1)^p bc \neq 0$. At this point, we can conclude similarly to the previous case: After permuting only the last $n$ rows and columns of $Q$,
we may assume with no loss of generality that $Q_0$ is its leading principal submatrix. In this setting, denoting by $S$ the Schur complement of $Q_0$ within $Q$, \cite[Remark 3.4]{DNZ2025} and \eqref{eq:genericdeterminantideals} imply
\[ \mathrm{rank} \ S  = \mathrm{rank} \ Q-(p+1) = \mathrm{rank} \ P(y) - (p+1) \geq n-2, \] and by \cite[Theorem 4.2]{DNZ2025} this implies in turn the existence of nonzero $h \times h$ minors of $Q$ (bordering $Q_0$) for every $h=p+2,\dots,p+n-1$. This yields the claim.
} 
\end{proof}

\subsection{Uniqueness of Smith ideals for minimal realizations over Hermite domains}

Let us define a \emph{Hermite domain} (or H-domain \cite{Lissner1965}) as an integral domain such that every right invertible matrix is completable to a unimodular matrix. {Equivalently \cite[Section I.4]{Lam2006}, it suffices to require this completion
property for right invertible row vectors. We note in passing that this terminology differs from that of Kaplansky
\cite{Kaplansky1949}, whose stronger notion of a Hermite
ring (called K-Hermite in \cite{Lam2006}) requires
diagonal reduction of every $1\times2$ and $2\times1$
matrix.}

 Theorem \ref{thm:allParethesame} states that, over a Hermite domain, even if Rosenbrock's Theorem fails then every minimal (in the sense of Theorem \ref{thm:rosenbrock}) system matrix fails in the same manner. 

\begin{theorem}\label{thm:allParethesame}
Let $\FR$ be a Hermite domain, and let $\K$ be the field of fractions of $\FR$. Fix $G \in \K^{m \times n}$ and suppose that
\[   P_1 = \begin{bmatrix}
-A_1 & B_1\\
C_1 & D_1
\end{bmatrix} \in \FR^{(m+p_1) \times (n+p_1)}, \qquad P_2 = \begin{bmatrix}
-A_2 & B_2\\
C_2 & D_2
\end{bmatrix} \in \FR^{(m+p_2) \times (n+p_2)} \]
satisfy $\det(A_1) \neq 0 \neq \det(A_2)$, $D_1 + C_1 A_1^{-1} B_1 = G=D_2+C_2 A_2^{-1} B_2$, and
\[  \mathcal{I}_{p_1}\left(\begin{bmatrix}
-A_1\\
C_1
\end{bmatrix}\right) = \mathcal{I}_{p_1}\left( \begin{bmatrix}
-A_1 & B_1
\end{bmatrix} \right) = \FR = \mathcal{I}_{p_2}\left(\begin{bmatrix}
-A_2\\
C_2
\end{bmatrix}\right) = \mathcal{I}_{p_2}\left( \begin{bmatrix}
-A_2 & B_2
\end{bmatrix} \right).   \]
Then:
\begin{enumerate}
\item Up to appending an initial sequence $\FR,\dots,\FR$ to the shorter lists, the Smith ideals of $P_1$ and $P_2$ coincide and the Smith ideals of $A_1$ and $A_2$ coincide;
\item $\mathcal{S}_k(A_i)=\FR$ for all $k=1,\dots,p_i-\min\{ m,n \}$ and $i=1,2$;
\item $\mathcal{S}_k(P_i)=\FR$ for all $k=1,\dots, p_i$ and $i=1,2$.
\end{enumerate}

\end{theorem}
\begin{proof}
Suppose with no loss of generality $m \geq n$ (up to transposing $G$, $P_1$ and $P_2$) and $p_1 \geq p_2$ (up to switching $P_1$ and $P_2$). It is clear that it suffices to prove the statement for $P_1$ and $P_3 = I_{p_1-p_2} \oplus P_2$, and hence we may assume $p_1=p_2=:p$.

Since $\begin{bmatrix}
-A_1 & B_1
\end{bmatrix}$ is right invertible and $\FR$ is Hermite, it can be completed to a unimodular matrix $U \in \FR^{(p+n) \times (p+n)}$. Partition 
$ \displaystyle U^{-1} = \begin{bmatrix}
U_{11} & U_{12}\\
U_{21} & U_{22}
\end{bmatrix} \in \FR^{(n+p) \times (n+p)}$, where $U_{11} \in \FR^{p \times p}, U_{22} \in \FR^{n \times n}$, and observe that defining $Z_1:=C_1 U_{12} + D_1 U_{22}$ we have
\[  \begin{bmatrix}
I_p & 0\\
-C_1 U_{11} - D_1 U_{21} & I_m
\end{bmatrix}      P_1 U^{-1} = \begin{bmatrix}
I_p & 0\\
0 & Z_1
\end{bmatrix}.\]
By Theorem \ref{thm:unimodularinvariance}, this shows that the Smith ideals of $P_1$ are $\FR,\dots,\FR, \mathcal{S}_1(Z_1),\dots,\mathcal{S}_n(Z_1)$. Moreover, we have the unimodular equivalence
\[ \begin{bmatrix}
    I_p & 0\\
    -U_{21} & I_n
\end{bmatrix} \begin{bmatrix}
    I_p & B_1\\
    0 & I_n
\end{bmatrix}   \begin{bmatrix}
    -A_1 & 0\\
    0 & I_n
\end{bmatrix}  U^{-1}  =  \begin{bmatrix}
    I_p & 0\\
    0 & U_{22}
\end{bmatrix}  ,\]
which implies via Theorem \ref{thm:unimodularinvariance} that $\mathcal{S}_{p-k}(A_1)=\mathcal{S}_{n-k}(U_{22})$ for all $k=0,\dots, \min\{p,n\}-1 $. 
If $n \geq p$, this suffices to characterize the Smith ideals of $A_1$. If instead $n < p$, we can observe that $\displaystyle \FR = \mathcal{I}_p (\begin{bmatrix}
-A_1 & B_1
\end{bmatrix}) \subseteq \mathcal{I}_{p-n}(A_1)$, and hence $\mathcal{S}_1(A_1)=\dots=\mathcal{S}_{p-n}(A_1)=\FR$.

We now relate $Z_1$ and $G$. Since $A_1 U_{12} = B_1U_{22}$ and $\det(A_1) \neq 0$, if $x \in \K^{n}$ satisfies $U_{22} x = 0$ then $U_{12} x = 0$, implying in turn $\displaystyle U^{-1}\begin{bmatrix}
0\\
x
\end{bmatrix}=0$ and hence $x=0$. It follows that $U_{22}$ is invertible over $\K$ and $G=Z_1 U_{22}^{-1}$.

By applying the exact same arguments to $P_2$, we can find $V_{12} \in \FR^{p \times n},V_{22} \in \FR^{n \times n}$ and $Z_2 = C_2 V_{12} + D_2 V_{22} \in \FR^{m \times n}$ such that (i) $\begin{bmatrix}
V_{12}\\
V_{22}
\end{bmatrix}$ is a submatrix of a unimodular matrix $V^{-1}$ (ii) the Smith ideals of $P_2$ are equal to $\FR,\dots,\FR,$ $\mathcal{S}_1(Z_2),\dots,\mathcal{S}_{n}(Z_2)$  (iii) the Smith ideals of $A_2$ are determined by those of $V_{22}$ in the same manner as for $A_1$ and $U_{22}$ (iv) $G=Z_2 V_{22}^{-1} = Z_1 U_{22}^{-1}$. We claim that there exists a unimodular matrix $W \in \FR^{n \times n}$ such that $V_{22} =U_{22} W$ and $Z_2 = Z_1 W$, implying the statement by invoking Theorem \ref{thm:unimodularinvariance}.

It remains to prove the claim. Observe first that by (i) there exist $X_1 \in \FR^{n \times p},X_2 \in \FR^{n \times n}$ satisfying $X_1 V_{12} + X_2 V_{22} = I_n$. Similarly, by Proposition \ref{prop21}, there exist matrices $X_3 \in \FR^{p \times p},X_4 \in \FR^{p \times m}$ such that $-X_3 A_2 + X_4 C_2 = I_p$. Hence, using $A_2 V_{12} = B_2 V_{22}$,
\[  I_n =  X_1 (X_4 C_2 - X_3 A_2) V_{12} + X_2 V_{22} = X_1 X_4 C_2 V_{12} + (X_2 - X_1 X_3 B_2) V_{22}  = X_5 Z_2 + X_6 V_{22},    \]
where in the last step we have added and subtracted $X_1 X_4 D_2 V_{22}$ and defined $X_5: = X_1 X_4, X_6:=X_2 - X_1 X_3 B_2 - X_1 X_4 D_2$. A similar argument shows $X_7 Z_1 + X_8 U_{22}=I_n$ for some matrices $X_7,X_8$ having elements in $\FR$. Define now $W = U_{22}^{-1} V_{22}$. Then, $Z_1 U_{22}^{-1} = Z_2 V_{22}^{-1}$ implies $Z_2 = Z_1 W$. Moreover, 
\[  W = (X_7 Z_1 + X_8 U_{22}) W = X_7 Z_2 + X_8 V_{22} \in \FR^{n \times n}\]
and
\[ W^{-1} = (X_5 Z_2 + X_6 V_{22}) W^{-1} = X_5 Z_1 + X_6 U_{22} \in \FR^{n \times n}.\]  
\end{proof}

\begin{remark}
The Quillen-Suslin Theorem \cite{Quillen1976, Suslin1976} implies that $\mathbb{D}[x_1,\dots,x_d]$ is a Hermite domain for every positive integer $d$ and every principal ideal domain $\mathbb{D}$. Hence, Theorem \ref{thm:allParethesame} applies in particular to polynomial rings that are of interest in engineering applications. Although our proof of Theorem \ref{thm:allParethesame} heavily relies on the fact that a right invertible matrix is completable to a unimodular matrix, we have not been easily able to construct a counterexample over an integral domain which is not Hermite. It is an interesting question for future research to determine whether Theorem \ref{thm:allParethesame} holds for a wider class of integral domains or not.
\end{remark}

\section*{Acknowledgements}

A scientific discussion that took place at the ILAS 2026 conference, and was initiated by Victor Gosea, motivated me to study a possible generalization of Rosenbrock's Theorem to rings where the Smith form is not guaranteed to exist. This conversation also involved Athanasios Antoulas, Etna Lindy, and Paul Van Dooren. I also thank Laura Cossu, Froil\'{a}n Dopico and Alfred Geroldinger for some comments on a preliminary version of this document. {Finally, I am grateful to an anonymous reviewer who shared very useful remarks.}

\end{document}